\documentclass[a4paper,english,final]{llncs}

\usepackage{amssymb,amsmath,enumerate,showkeys}
\usepackage{fancyhdr}
\usepackage{stmaryrd}
\usepackage{xspace}

\newcommand{\refthm}[1]{Theorem~\ref{#1}}
\newcommand{\refcor}[1]{Corollary~\ref{#1}}

\newcommand{\reflem}[1]{Lemma~\ref{#1}}
\newcommand{\refprop}[1]{Propositon~\ref{#1}}

\newcommand{\refsec}[1]{Section~\ref{#1}}
\newcommand{\refex}[1]{Example~\ref{#1}}

\newcommand{\set}[2]{\left\{#1\mathrel{\left|\vphantom{#1}\vphantom{#2}\right.}#2\right\}}
\newcommand{\oneset}[1]{\left\{\mathinner{#1}\right\}}

\newcommand{\abs}[1]{\left|\mathinner{#1}\right|}
\newcommand{\Abs}[1]{\left\Vert\mathinner{#1}\right\Vert}

\newcommand{\floor}[1]{\left\lfloor\mathinner{#1} \right\rfloor}

\newcommand{\N}{\mathbb{N}}
\newcommand{\Z}{\mathbb{Z}}

\newcommand{\R}{\mathbb{R}}

\renewcommand{\phi}{\varphi}
\newcommand{\eps}{\varepsilon}

\newcommand{\alp}{\alpha}

\newcommand{\sig}{\sigma}

\newcommand{\Sig}{\Sigma}
\newcommand{\Gam}{\GG}

\newcommand\GG{\Gamma}

\newcommand{\Oh}{\mathcal{O}}

\newcommand\ra{\longrightarrow}

\newcommand\RAS[2]{\overset{#1}{\Longrightarrow}_{#2}}

\newcommand{\smalloverline}[1]
{{\mspace{1mu}\overline{\mspace{-1mu}#1\mspace{-1mu}}\mspace{1mu}}}
\newcommand{\ov}[1]{\smalloverline{#1}}
\newcommand{\oi}[1]{{#1}^{-1}}

\newcommand{\inv}[1]{\ov{#1}} 

\newcommand{\ol}[1]{\overline{#1}}
\newcommand{\wh}[1]{\widehat{#1}}

\newcommand\GL{\mathop\mathrm{GL}}

\newcommand{\FIM}{\ensuremath{\mathrm{FIM}}}

\newcommand\DG{\mathop\mathrm{DG}}

\newcommand\IRR{\mathop\mathrm{IRR}}

\newcommand{\IFF}{if and only if\xspace}

\newcommand{\Ct}{Coxeter-trace\xspace}
\newcommand{\raCG}{right-angled {C}oxeter group\xspace}

\newcommand{\mazu}{{M}azur\-kie\-wicz\xspace}

\newcommand\lds{,\ldots ,} 
 
\newcommand{\sse}{\subseteq}
\newcommand{\es}{\emptyset}
\newcommand{\sm}{\setminus}
\newcommand{\os}[1]{\oneset{#1}}

\newenvironment{vd}{\noindent\color{blue} VD }{}

\newenvironment{jk}{\noindent\color{red} JK }{}

\newenvironment{ml}{\noindent\color{magenta} ML }{}

\newenvironment{jl}{\noindent\color{blue} JL : }{}

\begin{document}
\title{Logspace computations  in {C}oxeter groups and graph
groups} 
\author{Volker Diekert\inst{1} \and 
Jonathan Kausch\inst{1} \and Markus Lohrey\inst{2}}
\institute{
 FMI, Universit\"at Stuttgart, Germany 
\and  Insitut f\"ur Informatik, Universit\"at Leipzig,  Germany} 

\maketitle

\begin{abstract}
Computing normal forms in groups (or monoids) is in general harder
than solving the word problem (equality testing). However, normal form
computation has a much wider range of applications.
It is therefore interesting to investigate 
the complexity of computing normal forms  for important classes of groups.

For Coxeter groups we show that the following algorithmic tasks can 
be solved by a deterministic Turing machine using logarithmic work space, only: 
1. Compute the length of any geodesic normal form. 
2. Compute the set of letters occurring in any geodesic normal form. 
3. Compute the Parikh-image of any geodesic normal form in case that all 
defining relations have even length (i.e., in even Coxeter groups.) 
4. For right-angled  Coxeter groups we can do actually compute the short length normal form in logspace. (Note that short length normal forms are geodesic.) 

Next, we apply the results to right-angled Artin groups. They are also known as 
free partially commutative groups or as graph groups. As a consequence of our result 
on right-angled  Coxeter groups we show that shortlex normal forms 
in  graph groups can be
computed in logspace, too. Graph groups play an important r{\^o}le 
in group theory, and they 
have a close connection to concurrency theory. 
As an application of our results we show that the word problem for free
partially commutative inverse monoids is in logspace. This result generalizes 
a result of Ondrusch and the third author on free inverse monoids.
Concurrent systems which are deterministic and co-deterministic
can be studied via inverse monoids. 

%


\end{abstract}

\section{Introduction}
The study of group theoretical decision problems,
like the word problem (Is a given word equal to $1$ in the group?),
the conjugacy problem (Are two given words conjugated in the group?), 
and the isomorphism problem (Do two given group presentations yield
isomorphic groups?), is a classical topic 
in combinatorial group theory with a long history dating back to the beginning 
of the 20th century, see the survey \cite{Miller} for more 
details.

With the emergence of computational complexity theory, the complexity 
of these decision problems in various classes of groups has developed into 
an active research area, where algebraic methods as 
well as computer science techniques  complement one another in a
fruitful way.

In this paper we are interested in group theoretical  problems
which can be solved efficiently in parallel. 
More precisely, we are interested in {\em deterministic logspace}, called simply 
{\em  logspace} in the following. Note that logspace is at a lower level in the 
$\mathsf{NC}$-hierarchy of  parallel complexity classes:\footnote{$\mathsf{NC}^i$
denotes the class of languages that can be accepted by (uniform) boolean circuits
of polynomial size and depth $O(\log^i(n))$, where all gates have fan-in $\leq 2$,
see \cite{Vollmer99} for more details. We will not use the $\mathsf{NC}$-hierarchy in the rest
of this paper.}
$$
\mathsf{NC}^1 \subseteq \mathsf{LOGSPACE} \subseteq \mathsf{NC}^2 \subseteq \mathsf{NC}^3
\subseteq \cdots \subseteq \mathsf{NC} = \bigcup_{i \geq 1} \mathsf{NC}^i \subseteq \mathsf{P}
$$
It is a standard conjecture in complexity theory  that $\mathsf{NC}$ is strictly contained in $\mathsf{P}$.

A fundamental result
in this context 
was  shown 
in \cite{lz77,Sim79}: The word problem of finitely generated linear groups 
belongs to logspace. In \cite{lz77}, Lipton and Zalcstein proved
this result for fields of characteristic $0$. The case of a
field of prime characteristic was considered in \cite{Sim79} by Simon.
The class of groups with a word problem in
logspace is further investigated in \cite{Waack81}.   
Another important result is Cai's $\mathsf{NC}^2$ algorithm
for the word problem of a hyperbolic group \cite{cai92stoc}. 
In \cite{Lo05ijfcs} this result was improved to {\sf LOGCFL}, which is
the class of all languages that are logspace-reducible to a
context-free language. {\sf LOGCFL} is a subclass of $\mathsf{NC}^2$ and hence
in the intersection of the class  of problems which can be decided
in polynomial time and the class of problems which can be decided in space
$\log^2(n)$. As a parallel complexity class {\sf LOGCFL} coincides with the (uniform) class 
${\sf SAC}^1$.

Often, it is not enough to solve the word problem, but one has to compute a normal form for 
a given group element. 
Fix a finite generating set $\Gamma$ (w.l.o.g. closed under inverses) 
for the group $G$. Then, a {\em geodesic} for $g \in G$ is a shortest word
over $\Gamma$ that represents $g$. By choosing the lexicographical smallest (w.r.t. a fixed
ordering on $\Gamma$) word among all geodesics for $g$, one obtains the 
{\em shortlex normal form} of $g$. The problem of computing geodesics and 
various related problems were studied in 
\cite{DLS,elder2010,ElRe2010,MyRoUsVe08,PRaz}. It turned out that there are groups with 
an easy word problem (in logspace), but where simple questions related to 
geodesics are computationally hard. For instance, every  metabelian
group embeds (effectively)
into a direct product of linear groups; hence its word problem can be solved in logspace.
On the other hand, it is shown in \cite{DLS}, that the question whether a given element $x$
of the wreath product $\Z/ 2 \Z \wr (\mathbb{Z} \times \mathbb{Z})$ (a metabelian group)
has geodesic length at most $n$ is {\sf NP}-complete.
A corresponding result was shown in \cite{MyRoUsVe08} for the free metabelian group of rank 2.
Clearly, these results show that in general one cannot compute shortlex
normal forms in metabelian groups in polynomial time (unless $\mathsf{P} = \mathsf{NP}$).
On the positive side, for {\em shortlex automatic groups} \cite{ech92} (i.e., automatic groups, where
the underlying regular set of representatives is the set of shortlex normal forms) 
shortlex normal forms can be computed in quadratic time. 
Examples of shortlex automatic groups are Coxeter groups, Artin groups
of large type, and hyperbolic groups. So for all these classes, 
shortlex normal forms can be computed in quadratic time.
In \cite{MyRoUsVe08}, it is also noted
that geodesics in nilpotent groups (which are in general not
automatic) can be computed in polynomial time.

In this paper, we deal with the problem of 
computing geodesics and shortlex normal forms in logspace.
A function can be computed in logspace, if it can be computed
by a deterministic \emph{logspace transducer}. The latter
is a Turing machine with three tapes: (i) a read-only input 
tape, (ii) a read/write work tape of length 
$\Oh(\log n)$, and
(iii) a write-only output tape. The output is written sequentially
from left to right onto the output tape.
Every  logspace transducer can be transformed into an equivalent 
deterministic polynomial time machine. Still better, it can be 
simulated by a Boolean circuit of polynomial size and $\Oh(\log^2 n)$
depth. Although it is not completely obvious,  the class of logspace computable functions
is closed under composition.  (See e.g. 
the textbook \cite{Papa} for these facts.)

Recently, the class of groups, where geodesics and shortlex normal forms 
can be computed in logspace, attracted attention, see \cite{ElderElstonOstheimer11},
where it was noted among other results that shortlex normal forms in free groups
can be computed in logspace. (Implicitly, this result was also shown in  \cite{LohOnd07}.)
In this paper, we deal with the problem of computing shortlex normal
forms for Coxeter groups. Coxeter groups
are discrete reflection groups and play an important role in many
parts of mathematics, see \cite{bjofra05,davis08}.
Every Coxeter group is linear and therefore has a logspace
word problem  \cite{bjofra05,davis08}. Moreover, as mentioned above,
Coxeter groups are shortlex automatic
\cite{BrinkHowlett93,Casselman94}. Therefore shortlex normal forms
can be computed in quadratic time.
However, no efficient parallel algorithms are known so far. 
In particular, it is open whether 
shortlex normal forms in Coxeter groups  can be computed in logspace.
We do not solve this problem in this paper, but we are able to compute in logspace
some important invariants of geodesics. 
More precisely, we are able to compute in logspace 
(i) the length of the  shortlex normal form of a given element 
(\refthm{thm-geodesic-length})  and (ii) the alphabet of symbols that appear in the shortlex normal form
of a given element (\refthm{thm:allcoxoderwas}).
The proofs for these results combine non-trivial results for 
Coxeter groups with some advanced tools from computational algebra. 
More precisely, we use the
following results:
\begin{itemize}
\item 
The Chinese remainder representation of a given number $m$ (which is 
the tuple of remainders $m \text{ mod } p_i$ for the first $k$ primes $p_1, \ldots,
p_k$, where $m< p_1 p_2 \cdots p_k$)
can be transformed in logspace into the binary representation of $m$ 
\cite{ChDaLi01,HeAlBa02}. This result is the key for proving that iterated
multiplication and division can be computed in logspace.
\item
Arbitrary algebraic constants can be approximated in logspace
up to polynomially many bits. This result was recently shown in 
\cite{DaPra11,Jer11}.
\end{itemize}
For the case of even Coxeter groups, i.e., Coxeter groups 
where all defining relations have even length, we can combine 
\refthm{thm-geodesic-length} and \refthm{thm:allcoxoderwas} in
one more general result, saying that the Parikh-image of the shortlex
normal form can be computed in logspace (\refthm{thm-parikh}).
The Parikh-image of a word $w \in \Sigma^*$ is the image of $w$
under the canonical homomorphism from $\Sigma^*$ to $\N^{|\Sigma|}$.

As mentioned above, it remains open, whether shortlex normal forms
in Coxeter groups can be computed in logspace. 
In the second part of this paper, we prove that for the 
important subclass of {\em right-angled Coxeter groups}
shortlex normal forms can be computed in logspace (\refthm{thm:lexshort}). 
A right-angled Coxeter group is defined by a finite
undirected graph $(\Sigma, I)$ by taking $\Sigma$ as the set of group generators 
and adding the defining relations $a^2=1$ for all $a \in \Sigma$ and 
$ab=ba$ for all edges $(a,b) \in I$.
We use techniques from the theory of Mazurkiewicz traces \cite{die90}. More precisely, we
describe right-angled Coxeter groups by strongly confluent length-reducing 
trace rewriting systems.  Moreover, using the geometric representation
of right-angled  Coxeter groups, we provide an elementary proof that
the alphabet of symbols that appear in a geodesic for $g$ can be computed in logspace
from $g$ (\refcor{cor:nixoderwas}).\footnote{In contrast, the proof of 
\refthm{thm:allcoxoderwas}, which generalizes \refcor{cor:nixoderwas}
to all Coxeter groups, is more difficult in the sense that it uses
geometry and more facts {}from  \cite{bjofra05}.} In contrast
to general Coxeter groups, for right-angled
Coxeter groups this alphabetic information suffices in order to
compute shortlex normal forms in logspace. 

Right-angled Coxeter groups are tightly related to graph groups,
which are also known as \emph{free partially commutative groups} or as 
\emph{right-angled Artin groups}. A graph group 
is defined by a finite undirected graph $(\Sigma, I)$ by taking $\Sigma$ as the set of group generators 
and adding the defining relations $ab=ba$ for all edges $(a,b) \in I$.
Hence, a right-angled Coxeter group is obtained from a graph group
by adding all relations $a^2=1$ for all generators $a$.
Graph groups received in recent years a lot of attention in group
theory because of their rich subgroup structure \cite{BesBr97,CrWi04,GhPe07}. On the algorithmic 
side, (un)decidability results were obtained for many important
group-theoretic decision problems in graph groups 
\cite{CrGoWi09,dm06}. There is a standard embedding
of a graph  group into a right-angled Coxeter group \cite{hw99}. Hence,
also graph  groups are linear and have logspace word problems. 
Using the special properties of this embedding, we can show that
also for graph groups, shortlex normal forms can be computed in
logspace (\refthm{thm:lexshort}). We remark that this is 
an optimal result in the sense that logspace is the smallest known 
complexity class for the word problem in free groups already.
Clearly, computing shortlex normal forms is at least as difficult
than solving the word problem.

Finally, we apply  \refthm{thm:lexshort} to {\em free partially commutative inverse monoids}. 
These monoids arise naturally in the context of deterministic 
and co-determin\-istic concurrent systems. This includes many real systems, 
because they can be viewed as  deterministic concurrent systems with \emph{undo}-operations.
In \cite{DiekertLohreyMiller08} it was shown that the word problem for a 
free partially commutative inverse monoid can be solved in time ${\cal O}(n \log(n))$. 
(Decidability of the word problem is due to Da Costa \cite{vel03}.) 
Using our logspace algorithm for computing shortlex normal forms in a
graph group, we can show that the word problem for a free partially commutative inverse monoid
can be solved in logspace (\refthm{space}). Again, with state-of-the
art techniques, this can be viewed as an optimal result. It also
generalizes a corresponding result for free inverse monoids 
{}from \cite{LohOnd07}. 
Let us emphasize that in order to obtain \refthm{space} we have to be 
able to compute shortlex normal forms in  graph groups in logspace;
knowing only that the word problem is in logspace would not have been sufficient
for our purposes.

Let us remark that for all our results it is crucial that the
group (resp., the free partially commutative inverse monoids) is fixed
and not part of the input. For instance, it is not clear whether for 
a given undirected graph $(\Sigma,I)$ and a word $w$ over $\Sigma \cup
\Sigma^{-1}$ one can check in logspace whether $w = 1$ in the graph
group defined by the graph $(\Sigma,I)$.

The work on this paper started at the AMS Sectional Meeting, Las Vegas, May 2011, and 
was motivated by the lecture of Gretchen Ostheimer \cite{ElderElstonOstheimer11}. A preliminary version of our results appeared as a conference abstract at
the Latin American Symposium on Theoretical Informatics 
(LATIN 2012), \cite{dkl12latin}. In contrast to the conference  abstract this paper
provides full proofs and it contains new material about even Coxeter groups and 
how to compute geodesic lengths in all Coxeter groups. 

\section{Notation}
Throughout $\Sig$  (resp.{} $\Gam$) denotes a finite \emph{alphabet}. This is a finite set, sometimes 
equipped with a linear order. An element of $\Sig$ is called a \emph{letter}. By $\Sig^*$ 
we denote the free monoid over $\Sig$. 
For a word $w\in \Sig^*$ we denote by $\alp(w)$ the {\em alphabet of $w$}: 
it is the set of letters occurring in $w$. 
With $|w|$ we denote the length of $w$. The {\em empty word } has length $0$; 
and it is denoted by 1 as other neutral elements in monoids or groups. 

All groups and monoids $M$ in this paper are assumed to be finitely generated;
and they come with a monoid homomorphism
$\pi: \Sig^* \to M$. Frequently we  assume that $M$ comes with  an involution\footnote{An {\em involution} on a set 
$\Gamma$ is a permutation $a \mapsto \ov a$ such that $\ov{\ov a} = a$. An involution of a monoid satisfies in addition 
$\ov {xy}= \ov y\; \ov x$.} $x \mapsto \ov x$ on $M$, and then we require that  $\pi(\Sig) \cup \ov{\pi( \Sig)}$ generates $M$ as a monoid. 

If the monoid $M$ is a group $G$, then  the involution is always
given by taking inverses, thus $\ov x = \oi x$. Moreover, 
$G$ becomes a factor group of the \emph{free group} $F(\Sig)$ thanks to
$\pi: \Sig^* \to G$.

Let $\ov \Sig$ be a disjoint copy of $\Sig$ and $\Gam = \Sig \cup 
\ov \Sig$. There is a unique extension of the natural mapping 
$\Sig \to \ov \Sig: \, a \mapsto \ov a$
such that $\Gam^*$ becomes a monoid with involution: We let  $\ol{\ol a} = a$ 
and $\ov{a_1 \cdots a_n} = \ov {a_n} \cdots \ov {a_1}$.  Hence, we can
lift our homomorphism $\pi: \Sig^* \to M$ 
to a surjective monoid homomorphism $\pi: \Gam^* \to M$ which respects 
the involution, i.e., $\pi(\ov{x}) = x^{-1}$. 

Given a surjective homomorphism $\pi: \Gam^* \to M$ and a linear order on  $\Gam$ we can define 
the geodesic length and the shortlex normal form for elements in $M$ as follows. 
For $w \in M$, the \emph{geodesic length} $\Abs{w}$ is the length of a shortest 
word in $\pi^{-1}(w)$. The \emph{shortlex normal form} of $w$ is the 
lexicographical first word in the finite set 
$\set{u \in \pi^{-1}(w)}{\abs{u}= \Abs{w}}$. By a \emph{geodesic} 
we mean any word in the finite set $\set{u \in \pi^{-1}(w)}{\abs{u}= \Abs{w}}$.


\section{Coxeter groups}\label{sec:cox}
A {\em Coxeter group} $G$ is given 
by a generating set $\Sig = \os{a_1 \lds a_n}$ of $n$ generators 
and a symmetric  $n \times n $ matrix $M = (m_{i,j})_{1 \leq i , j
  \leq n}$ over $\N$ 
such that $m_{i,j} = 1 \iff i = j$. 
The defining relations are 
$(a_i a_j)^{m_{i,j}} = 1$ for $1 \leq i,j \leq n$. 
In particular, $a_i^2 = 1$ for $1 \leq i \leq n$. 
Traditionally, one writes the entry $\infty$ instead $0$  in the Coxeter matrix $M$ 
and then $m_{i,j}$ becomes the order of the element $a_ia_j$.

A Coxeter group is called \emph{even}, if all $m_{i,j}$ are even numbers for $i \neq j$. 
It is called \emph{right-angled}, if $m_{i,j} \in \{0,1,2\}$ for all $i,j$. 
The defining relations of a right-angled Coxeter group can be rewritten in the 
following form: $a_i^2 = 1$ for $1 \leq i \leq n$ and $a_i a_j = a_j a_i$ for
$(i,j)\in I$ where $I$ denotes a symmetric and irreflexive relation 
$I \sse \{1,\ldots ,n\} \times \{1,\ldots ,n\}$. Thus, one could say that 
a right-angled Coxeter group is a \emph{free partially commutative} 
Coxeter group. Readers interested only in right-angled Coxeter groups or in the application to graph groups (i.e., 
right-angled Artin groups) may proceed directly to  \refsec{mazu}.

\subsection{Computing the geodesic alphabet and geodesic length}\label{sec:allcox}
Throughout this subsection $G$ denotes a Coxeter group given by a fixed  $n \times n$ 
Matrix $M$ as above. 
One can show that if $u$ and $v$ are geodesics with $u = v$ in $G$
then $\alpha(u) = \alpha(v)$ \cite[Cor.~1.4.8]{bjofra05}. (Recall that 
$\alpha(x)$ denotes the alphabet of the word $x$.) 
We will show how to compute this
alphabet in logspace. Moreover, we will show that also the geodesic length
$\abs{w}$ for a given $w \in G$ can be computed in logspace.

Let $\R^{\Sig}$ be the $n$ dimensional real vector space where
the letter $a$ is identified with the $a$-th unit vector. 
Thus, vectors will be written 
as formal sums $\sum_{b\in \Sig} \lambda_b b$
with real coefficients $\lambda_b\in \R$.
We fix the standard geometric representation $\sig: G \to \mathrm{GL}(n,\R)$,
where we write $\sig_w$ for the mapping $\sig(w)$,
see e.g. \cite[Sect.~4.2]{bjofra05}:
\begin{equation} \label{eq-geo-repr}
\sigma_{a_i}({a_j}) = 
\begin{cases}
{a_j} + 2 \cos(\pi/m_{i,j}) \cdot {a_i}  & \text{ if } m_{i,j} \neq 0 \\
{a_j} + 2  \cdot {a_i}  & \text{ if } m_{i,j} = 0  
\end{cases}
\end{equation}
Note that for $a \in \Sigma$, $\sig_w(a)$ cannot be the zero vector, since $\sig_w$ is invertible.
We write $\sum_{b\in \Sig} \lambda_b b \geq 0$ 
if $\lambda_b \geq 0$ for all $b\in \Sig$. 
The following fundamental lemma can be found in \cite[Prop. 4.2.5]{bjofra05}:

\begin{lemma}\label{lem:geqzero}
Let $w \in G$ and  $a\in \Sig$. 
We have 
$$
\Abs{wa} = \begin{cases}
 \Abs{w}+1 & \text{ if }  \sig_w(a) \geq 0 \\
\Abs{w}-1 & \text{ if }   \sig_w(a) \leq 0
\end{cases}
$$
\end{lemma}

\begin{lemma} \label{lemma-main}
For a given $w \in G$ and  $a,b\in \Sig$, one can check in logspace, whether 
$\lambda_b \geq 0$, where $\sum_{b\in \Sig} \lambda_b b = \sigma_w(a)$.
\end{lemma}
In order to prove \reflem{lemma-main}, we need several tools.
Let $p_i$ denote the $i^{\text{th}}$ prime number.
It is well-known from number theory that the $i^{\text{th}}$ prime 
requires $O(\log(i))$ bits in its binary representation.
For a number $0 \leq M < \prod_{i=1}^m p_i$ we define the 
{\em Chinese remainder representation} $\mathsf{CRR}_m(M)$ as the 
$m$-tuple 
$$
\mathsf{CRR}_m(M) = (M \text{ mod } p_i)_{1 \leq i \leq m} .
$$
By the Chinese remainder theorem, the mapping 
$M \mapsto \mathsf{CRR}_m(M)$ is a bijection from 
the interval $[0, \prod_{i=1}^m p_i-1]$
to $\prod_{i=1}^m [0,p_i-1]$. 
By the following theorem, we can transform a {\sf CRR}-representation very
efficiently into binary representation.

\begin{theorem}[{\cite[Thm.~3.3]{ChDaLi01}}] \label{theorem hesse und co}
For a given tuple $(r_1,\ldots, r_m) \in \prod_{i=1}^m [0,p_i-1]$, we can compute
in logspace the binary representation of the unique number $M \in [0, \prod_{i=1}^m p_i-1]$
such that $\mathsf{CRR}_m(M) = (r_1, \ldots, r_m)$.
\end{theorem}
By \cite{HeAlBa02}, the transformation from the {\sf CRR}-representation to the binary
representation can be even computed by  $\mathsf{DLOGTIME}$-uniform
$\mathsf{TC}^0$-circuits. 
%
Our second tool is a gap theorem for values $p(\zeta)$,
where $p(x) \in \mathbb{Z}[x]$ and $\zeta$ is a root of unity.
For a polynomial $p(x) = \sum_{i=0}^n a_i x^i$ with integer
coefficients $a_i$ let $|p(x)| = \sum_{i=0}^n |a_i|$.
\begin{theorem}[{\cite[Thm.~3]{Litow2010}}] \label{thm:canny}
Let $p(x) \in \mathbb{Z}[x]$ and let $\zeta$ be a $d^{th}$ root of unity
such that $p(\zeta) \neq 0$. Then $|p(\zeta)| > |p(x)|^{-d}$.
\end{theorem}
Finally, our third tool for the proof of \reflem{lemma-main}
is the following result, which was recently shown (independently)
in \cite{DaPra11,Jer11}.

\begin{theorem}[{\cite[Thm~.2]{DaPra11}},{\cite[Cor.~4.6]{Jer11}}] \label{thm-approx-alg}
For every fixed algebraic number $\alpha\in\mathbb{R}$ the following
problem can be computed in logspace:

\noindent
INPUT: A unary coded number $n$. 

\noindent
OUTPUT: 
A binary representation of the integer $\floor{2^n\alp}$. 
\end{theorem}

\begin{remark}\label{rem:jer11}
The result of \cite{Jer11} is actually stronger showing that the output 
in \refthm{thm-approx-alg}
can be computed in uniform $\mathsf{TC}^0$. 
\end{remark}
{\em Proof of \reflem{lemma-main}.}
We decompose the 
logspace algorithm into several logspace computations.
The linear mapping $\sig_w$ can be written as a product
of matrices $A_1 A_2 \cdots A_{|w|}$, where every $A_i$ is an
$(n \times n)$-matrix with  entries from $\{0, 1,2\} \cup
\{2 \cos(\pi/m_{i,j})\mid m_{i,j}\neq 0 \}$ (which is the set of 
coefficients that appear in \eqref{eq-geo-repr}). Then, we have to check
whether this matrix product applied to the unit vector $a$ has a non-negative
value in the $b$-coordinate. This value is the entry $(A_1 A_2 \cdots
A_{|w|})_{a,b}$ of the product matrix $A_1 A_2 \cdots A_{|w|}$.

Let $m$ be the least common multiple of all $m_{i,j} \neq 0$; this is
still a constant.
Let $\zeta = e^{\pi i/m} $, which is a primitive $(2m)^{th}$ root of unity.
If $m = m_{i,j} \cdot k$,  we have 
$$
2 \cdot \cos \bigg( \frac{\pi}{m_{i,j}} \bigg) =   \zeta^k +
\zeta^{2m-k} .
$$
Hence, we can assume that every $A_i$ is an $(n \times n)$-matrix over  $\mathbb{Z}[\zeta]$.
We now replace $\zeta$ by a variable $X$ in all matrices
$A_1, \ldots, A_{|w|}$; let us denote the resulting matrices over 
the ring $\mathbb{Z}[X]$ with $B_1, \ldots, B_{|w|}$.
Each entry in one of these matrices is a polynomial of degree $<2m$
with coefficients bounded by $2$. More precisely,  for every entry $p(X)$ 
of a matrix $B_i$ we have $|p(X)| \leq 2$.
Let $|B_i|$ be the sum of all $|p(X)|$ taken over all entries of the matrix $B_i$.
Hence, $|B_i| \leq 2 n^2$.

\medskip
\noindent
{\em Step 1.}  
In a first step, we show that the product $B_1 \cdots B_{|w|}$ can be
computed in logspace in the ring 
$\mathbb{Z}[X]/(X^{2m}-1)$ (keeping in mind that $\zeta^{2m}=1$). 
Every entry in the product $B_1 \cdots B_{|w|}$ is a polynomial
of degree $<2m$ with coefficients bounded in absolute value  by 
$|B_1| \cdots |B_{|w|}| \leq (2n^2)^{|w|}$.
Here $n$ is a fixed constant.  Hence, every
coefficient in the matrix  $B_1 \cdots B_{|w|}$ can be represented 
with $O(|w|)$ bits. In logspace, one can compute a list of 
the first $k$ prime numbers $p_1, p_2, \ldots, p_k$, where $k \in O(|w|)$
is chosen such that $\prod_{i=1}^k p_i > (2n^2)^{|w|}$ \cite{ChDaLi01}.
Each $p_i$ is bounded by $|w|^{O(1)}$. 

For every $1 \leq i \leq k$, we can compute in logspace 
the matrix product $B_1 \cdots B_{|w|}$ in
$\mathbb{F}_{p_i}[X]/(X^{2m}-1)$, i.e., we compute the coefficient of each
polynomial in $B_1 \cdots B_{|w|}$ modulo $p_i$. In the language of
\cite{ChDaLi01}: For each  coefficient of a
polynomial in $B_1 \cdots B_{|w|}$, we compute its Chines remainder
representation. From this representation, we can compute in logspace
by \refthm{theorem hesse und co} the binary representation of the coefficient.
This shows that the product $B = B_1 \cdots B_{|w|}$ can be computed
in the ring $\mathbb{Z}[X]/(X^{2m}-1)$.

\medskip
\noindent
{\em Step 2.} We know that if $X$ is substitued by $\zeta$ in the
matrix $B$, then we obtain the product  $A = A_1 \cdots A_{|w|}$ (the
matrix we are actully interested in), which is 
a matrix over  $\mathbb{R}$. Every entry of the matrix
$A$ is of the form $\sum_{j=0}^{2m-1} a_j \zeta^j$, where
$a_j$ is a number with $O(|w|)$ bits that we have computed in Step 1. If 
$\sum_{j=0}^{2m-1} a_j \zeta^j \neq 0$, then by 
\refthm{thm:canny}, we have
$$
\left|\sum_{j=0}^{2m-1} a_j \zeta^j \right| > 
\left(\sum_{j=0}^{2m-1} |a_j|\right)^{-2m} .
$$
Since $m$ is a constant, and $|a_j| \leq 2^{O(|w|)}$, we have 
$$
\sum_{j=0}^{2m-1} a_j \zeta^j = 0 \quad \text{ or } \quad
\left|\sum_{j=0}^{2m-1} a_j \zeta^j \right| > 2^{-c|w|}
$$
for a constant $c$. Therefore, to check whether
$\sum_{j=0}^{2m-1} a_j \zeta^j \geq 0$ or
$\sum_{j=0}^{2m-1} a_j \zeta^j \leq 0$, it suffices
to approximate this sum up to $c|w|$ many fractional 
bits. This is the goal of the second step.

Since we are sure that $\sum_{j=0}^{2m-1} a_j \zeta^j \in \mathbb{R}$,
we can replace the sum symbolically by its real part, which is 
$\sum_{j=0}^{2m-1} a_j \cos(j\pi/m)$. In order to approximate 
this sum up to $c|w|$ many fractional bits, it suffices to approximate each
$\cos(j\pi/m)$ up to $d |w|$ many fractional bits (recall that $a_j \in
2^{O(|w|)}$), where the constant $d$ is large enough.

Every number $\cos(q \cdot \pi)$ for $q \in \mathbb{Q}$
is algebraic; this seems to be folklore and follows easily
from DeMoivre's formula 
($(\cos\theta+i\sin\theta)^n=\cos(n\theta)+i\sin(n\theta)$).
Theorefore, by \refthm{thm-approx-alg}, every number 
$\cos(j\pi/m)$ ($0 \leq j \leq 2m-1$) can be approximated in logspace
up to $d |w|$ many fractional bits. This concludes the proof.
\qed

\medskip

\noindent
\reflem{lemma-main} can be used in order to compute in logspace
the geodesic length $\Abs{w}$ for a given group element $w \in G$:

\begin{theorem} \label{thm-geodesic-length}
For a given word $w \in \Sigma^*$, the geodesic length $\Abs w$ can be computed
in logspace.
\end{theorem}

\begin{proof}
By \reflem{lem:geqzero}, the following algorithm correctly computes $\Abs w$
for $w = a_1 \cdots a_k$.

\medskip
\noindent
$\ell := 0$;\\
{\bf for} $i=1$ {\bf to} $k$ {\bf do} \\
\hspace*{.5cm} {\bf if} $\sig_{a_1 \cdots a_{i-1}}(a_i) \geq 0$ {\bf then}\\
\hspace*{1cm} $\ell := \ell+1$\\
\hspace*{.5cm} {\bf else}\\
\hspace*{1cm} $\ell := \ell-1$\\
\hspace*{.5cm} {\bf endif} \\
{\bf endfor} \\ 
{\bf return} $\ell$.

\medskip
\noindent
By \reflem{lemma-main} it can be implemented in logspace.
\qed
\end{proof}
We finally apply \reflem{lemma-main} in order to compute 
in logspace the set of all letters that occur in a geodesic for 
a given group element $w \in G$. As remarked before, this alphabet
is independent of the concrete geodesic for $w$.

Introduce a new letter $x\not\in\Sigma$ with $x^2 =1$,
but no other new defining relation. 
This yields the Coxeter group $G' = G * (\Z/ 2\Z)$ generated by $\Sig'= \Sig \cup \os{x}$.
Thus, $ax$ is of infinite order in $G'$ for 
all $a \in \Sig$. Clearly, $\Abs{wx} > \Abs{w}$ for all $w \in G$. Hence,
$\sig_w(x) \geq 0$ for all $w \in G$ by \reflem{lem:geqzero}.

\begin{lemma}\label{lem:x}
Let $w \in G$ and $\sig_w(x) = \sum_{b\in \Sig'} \lambda_b b $. 
Then for all $b \in \Sig$  we have $\lambda_b \neq 0$ \IFF the letter
$b$ appears in the shortlex normal form of $w$.
\end{lemma}

\begin{proof}
We may assume that $w$ is a geodesic in $G$. We prove the result
by induction on $\Abs w = |w|$. If $w=1$, then the assertion is trivial.
If $b \in \Sig$ does not occur as a letter in $w$, then it is clear that 
$\lambda_b = 0$. Thus, we may assume that $b \in \alpha(w)$ and we have to 
show that $\lambda_b \neq 0$.
By induction, we may write $w= ua$ with $\Abs{uax} > \Abs{ua } > \Abs{u}$. 
 We have $\sig_w(x) = \sig_u\sig_a(x) = \sig_u(x + 2a) = \sig_u(x) + 2\sig_u(a)$.
The standard geometric representation yields moreover
$\sig_w(x) = x + \sum_{c\in \Sig} \lambda_c c$, where
$\lambda_c \geq 0$ for all $c\in \Sig$ by \reflem{lem:geqzero}. 
As $\Abs{ua } > \Abs{u}$ we get $\sig_u(a)\geq 0$ by
\reflem{lem:geqzero}. 
Moreover, by induction
(and the fact $\Abs{ux} > \Abs{u}$), we know that for all letters
$c\in \alp(u)$ the corresponding coefficient
in $\sig_u(x)$ is strictly positive.
Thus, we are done if $b \in \alpha(u)$. 
So, the remaining case is that $b = a \not\in\alpha(u)$. 
However, in this case $\sig_u(a) = a + \sum_{c\in \Sig \setminus\{a\}} \mu_c c$. Hence $\lambda_a \geq 2$. 
\qed
\end{proof}
 
\begin{theorem}\label{thm:allcoxoderwas}
There is a logspace transducer which on input $w \in \Sig^*$ 
computes the set of letters occurring in the shortlex normal form 
of $w$. 
\end{theorem}

\begin{proof}
By \reflem{lem:x}, we have to check for every letter $b \in
\Sigma$, whether $\lambda_b = 0$, where $\sum_{b\in \Sig'} \lambda_b b =
\sig_w(x)$. By \reflem{lemma-main} (applied to the Coxeter group
 $G'$) this is possible in logspace.
\qed
\end{proof}
Let us remark that the use of \reflem{lemma-main} in the proof of
\refthm{thm:allcoxoderwas} can be avoided,
using the technique from \cite{lz77} and 
\reflem{lem:x}. 
Every $\lambda_b$ belongs to the ring 
$\Z[\zeta] \cong \Z[X]/\Phi(X)$, where  
$\zeta$ is a primitive $(2m)^{th}$ root of unity,
$\Phi(X)$ is the $(2m)^{th}$ cyclotomic polynomial, and
$m$ is the least common multiple of all $m_{i,j} \neq 0$.
In order to check, whether $\lambda_b = 0$, we 
can check whether the value is zero $\bmod \; r$ with respect to all $r$ 
up to a polynomial threshold.

\subsection{Computing the geodesic Parikh-image in even Coxeter groups}\label{sec:evencox}

In this section we assume that $G$ is an even Coxeter group. Thus,
the entries $m_{i,j}$ are even for all $i \neq j$. 

Let $a \in \Sigma$ be a letter and $w \in \Sigma^*$. 
By $\abs{w}_a$ we denote the number of $a$'s 
in a word $w\in \Gam^*$. The {\em Parikh-image} of $w$ is the 
vector $[\, \abs{w}_a\, ]_{a \in \Sigma} \in \N^\Sigma$.  
In other words, the Parikh-image of $w$ is the image of
$w$ under the canonical homomorphism from the free monoid
$\Sigma^*$ to the free commutative monoid $\N^{\Sigma}$.

We show that for even Coxeter groups, 
the Parikh-image of geodesics can be computed in logspace. 
Actually, all geodesics for a given group element of an
even Coxeter group have the same Parikh-image:

\begin{lemma}\label{welldef}
Let $G$ be an even Coxeter group and
let $u,v \in \Sigma^*$ be geodesics with $u=v$ in $G$.
Then we have $\abs{u}_a = \abs{v}_a$ for all $a \in \Sigma$.
\end{lemma}

\begin{proof}
Let $a,b \in \Gamma$ be  letters such that 
$(ab)^m=1$ for some $m\geq 2$. Since $G$ is even, all such values $m$ are even 
and we obtain the relation $(ab)^{m/2}=(ba)^{m/2}$ which does not effect the 
Parikh-image. Now, it follows from a well-known result about Tits' rules (c.f.~\cite{bjofra05}) that geodesics can be transformed into each other 
by using the relations  $(ab)^{m/2}=(ba)^{m/2}$, only. 
Consequently, $\abs{u}_a = \abs{v}_a$ for all $a \in \Sigma$.
\qed
\end{proof}

\begin{lemma}\label{lemma-parikh}
Let $G$ be an even Coxeter group, $a \in \Sigma$, and let $u, w$ be geodesics
such that $w a = u$ in $G$. Then there exists $\eps\in \{1,-1\}$
such that $|u| = |w|+\eps$ and $\abs{u}_a = \abs{w}_a+ \eps$. For all $b \in \Sigma\setminus\{a\}$ we have
$\abs{u}_b = \abs{w}_b$.
\end{lemma}

\begin{proof}
By \reflem{lem:geqzero} there exists $\eps\in \{1,-1\}$
with $|u| = |w|+\eps$.  Moreover, 
since $a^2=1$, we have $ua=w$ and $wa=u$ in $G$. 
Hence, if $|w| = |u|+1$ (resp., $|u| = |w|+1$), then 
$ua$ and $w$ (resp., $wa$ and $u$)
are geodesics defining the same group element in $G$.
\reflem{welldef} implies that $\abs{ua}_c = \abs{w}_c$ 
(resp., $\abs{wa}_c = \abs{u}_c$)
for all $c \in \Sigma$.
This implies the conclusion of the lemma.
\qed
\end{proof}

\begin{theorem} \label{thm-parikh}
Let $G$ be an even Coxeter group.
For a given word $w \in \Sigma^*$, the Parikh-image of the shortlex normal 
form for $w$ can be computed in logspace. 
\end{theorem}

\begin{proof}
\reflem{lemma-parikh} shows that the following straightforward 
modification of the logspace algorithm in (the proof of) \refthm{thm-geodesic-length}
computes the Parikh-image of the shortlex normal form for $w$ correctly.
Let  $w = a_1 \cdots a_k$ be the input word. 

 \medskip
\noindent
{\bf for all} $a \in \Gam$ {\bf do} $\ell_a := 0$;\\
{\bf for} $i=1$ {\bf to} $k$ {\bf do} \\
\hspace*{.5cm} {\bf if} $\sig_{a_1 \cdots a_{i-1}}(a_i) \geq 0$ {\bf then}\\
\hspace*{1cm} $\ell_{a_i} := \ell_{a_i}+1$\\
\hspace*{.5cm} {\bf else}\\
\hspace*{1cm} $\ell_{a_i} := \ell_{a_i}-1$\\
\hspace*{.5cm} {\bf endif} \\
{\bf endfor}\\
{\bf return} $[\, \ell_a\, ]_{a \in \Gam}$.
\qed
\end{proof}

\section{Mazurkiewicz traces and graph  groups}\label{mazu}

In the rest of the paper, we will deal with right-angled Coxeter
groups. As explained in Section~\ref{sec:cox}, a right-angled Coxeter
group can be specified by a finite undirected graph $(\Sigma,I)$.
The set $\Sigma$ is the generating set and the relations are 
$a^2=1$ for all $a \in \Sigma$ and $ab=ba$ for all $(a,b) \in I$.
Hence, $I$ specifies a partial commutation relation, and 
elements of a right-angled Coxeter group can be represented by partially
commutative words, also known as Mazurkiewicz traces. In this section, we 
will introduce some basic notions from the theory of 
Mazurkiewicz traces, see \cite{die90,dr95} for more details.

An \emph{independence alphabet} is a pair $(\Sig, I)$, where 
$\Sig$ is a finite set (or  \emph{alphabet}) and $I \sse \Sig \times \Sig$ 
is an irreflexive and symmetric relation, called
the \emph{independence relation}. Thus, $(\Sig, I)$ is a finite 
undirected graph.  The complementary relation $D = ( \Sig \times \Sig) \sm I$   
is called a \emph{dependence relation}. It is reflexive and symmetric. 
We extend $(\Sig, I)$ to a graph $(\Gam, I_\Gamma)$, where
$\Gam = \Sig \cup \ov \Sig$ with $\Sig \cap \ov \Sig = \emptyset$,
and $I_\Gamma$ is the minimal independence relation with $I \sse I_\Gamma$ 
and such that $(a,b)\in I_\Gamma$ implies $(a,\ov b)\in I_\Gamma$.
The independence alphabet $(\Sig, I)$ defines a 
\emph{free partially commutative monoid} (or {\em trace monoid}) $M(\Sig, I)$ 
and a \emph{free partially commutative group} $G(\Sig, I)$ by:
\begin{align*}
M(\Sig, I) &= \Sig^*/ \set{ab = ba}{(a,b) \in I}, \\ 
G(\Sig, I) &= F(\Sig)/ \set{ab = ba}{(a,b) \in I}.
\end{align*}
Free partially commutative groups are also known as 
{\em right-angled Artin groups} or {\em graph groups}.
Elements of $M(\Sig, I)$ are called \emph{(\mazu) traces}. They have a unique description
as {\em dependence graphs}, which are 
node-labelled acyclic graphs defined as follows. Let $u = a_1\cdots a_n \in \Sig^*$ be a word. 
The vertex set of the dependence graph $\DG(u)$ 
 is $\{1,\ldots,n\}$ and vertex $i$
is labelled with $a_i \in \Sig$. There is an arc 
from vertex $i$ to $j$ if and only if $i<j$ and 
$(a_i,a_j) \in D$. Now, two words define the same trace in $M(\Sig,I)$
if and only if their dependence graphs are isomorphic.
A dependence graph is acyclic, so its transitive closure
is a labelled partial order $\prec$, which can be uniquely 
represented by its {\em Hasse diagram}:
There is an arc from $i$ to $j$ in the Hasse diagram, if 
$i \prec j$ and there does not exist $k$ with $i \prec k \prec j$.

A trace $u\in M(\Sigma,I)$ is a \emph{factor} of $v\in M(\Sigma,I)$, 
if $v \in M(\Sigma,I) u M(\Sigma,I)$. 
The set of letters occurring in a trace $u$ is denoted by $\alp(u)$. 
The independence relation $I$ is extended to traces by letting $(u,v)\in I$, if 
$\alp(u) \times \alp(v) \sse I$. We also write 
$I(a) = \set{b \in \Sig}{(a,b) \in I}$. 
A trace $u$ is called a \emph{prime} if $\DG(u)$  has exactly one maximal element.
Thus, if $u$ is a prime, then we can write $u$ as $u=va$ in $M(\Sig,I)$,
where $a\in \Sig$ and $v \in M(\Sig,I)$
are uniquely defined. Moreover, this property characterizes primes. 
A \emph{prime prefix} of a trace $u$ is a prime trace $v$ such that
$u = vx$ in $M(\Sigma,I)$ for some trace $x$.
We will use the following simple fact. 

\begin{lemma} \label{lemma-prim-prefixes}
Let $(\Sigma,I)$ be a fixed independence relation. There is a logspace
transducer that on input $u  \in M(\Sigma,I)$ outputs a list
of all prime prefixes of $u$.
\end{lemma}

\begin{proof}
The prime prefixes of $u$ correspond to the downward-closed subsets
of the dependence graph $\DG(u)$ that have a unique maximal element.
Assume that $u = a_1 a_2 \cdots a_n$ with $a_i \in \Sigma$. 
Our logspace transducer works in $n$ phases. In the $i$-th phase
it outputs the sequence of all symbols $a_j$ ($j \leq i$) such that
there exists a path in $\DG(u)$ from $j$ to $i$. 
Note that there exists a path from $j$ to $i$ in $\DG(u)$ if and only
if there is such a path of length at most $|\Sigma|$. Since $\Sigma$
is fixed, the existence of such a path can be checked in logspace
by examining all sequences $1 \leq i_1 < i_2 < \cdots < i_k = i$ with $k \leq
|\Sigma|$. Such a sequence can be stored in logarithmic space since
$|\Sigma|$ is a constant.
\qed
\end{proof}
We use standard notations from the theory of rewriting systems, cf \cite{bo93springer}.
Let $M = M(\Sig, I)$. A  \emph{trace rewriting system} is a finite set  
of rules $S \sse M\times M$. A rule is often written in the form
$\ell \ra r$. The system  $S$ defines a one-step rewriting relation
$\RAS{}{S} \; \sse M \times M$ by 
$x\RAS{}{S}y$ if there exist $(\ell,r)\in S$ and $u,v \in M$ with $x =
u\ell v$ and $y = ur v$ in $M$.
By $\RAS{*}{S}$, we denote the reflexive and transitive closure of 
$\RAS{}{S}$. The set $\IRR(S)$ denotes the set of traces to which no rule of $S$ 
applies. 
If $S$ is confluent and terminating, then
for every $u \in M$ there is a unique $\wh u \in \IRR(S)$ 
with $u \RAS{*}{S} \wh u$, and 
$\IRR(S)$ is a set of normal forms for the quotient monoid $M/S$. 
If, in addition,  $S$ is length-reducing (i.e., $|\ell| > |r|$ for all
$(\ell,r) \in S$), then
$\Abs{\pi(u)} = \abs{\wh u}$ for the canonical homomorphism $\pi: M \to M/S$.

\begin{example}\label{ex:gg}
The system $S_G = \set{a \ov a \ra 1}{a \in \Gam}$ is (strongly)  
confluent and length-reducing over $M(\Gam, I_\Gam)$ \cite{die90}. The quotient 
monoid $M(\Gam, I_\Gam)/S_G$ is the graph group $G(\Sig, I)$.
\end{example}
By \refex{ex:gg} elements in  {graph group}s have a unique description
as {\em dependence graphs}, too. 
A trace belongs to $\IRR(S_G)$ if and only if it does not contain a factor 
$a \ov a $ for $a \in \Gamma$. In the dependence graph, this means that the Hasse diagram does not contain
any arc from a vertex labeled $a$ to a vertex labeled $\ov a$ with $a \in \Gam$.  
Moreover, a word $u \in \Gamma^*$ represents a trace from 
$\IRR(S_G)$ if and only if $u$ does not contain a factor of the form
$a v \ov a$ with $a \in \Gam$ and $\alpha(v) \subseteq I(a)$.

\section{Right-angled {C}oxeter groups}\label{sec:c}
Some of the results on right-angled Coxeter groups 
in this section are covered by more general statements 
in \refsec{sec:cox}. However, the former section used quite involved tools from 
computational algebra and an advanced theory of Coxeter groups. In contrast the
results we prove here on right-angled Coxeter groups are purely combinatorial. 
Hence we can give simple and elementary proofs which makes this section fully self-contained.  
Moreover, in contrast to the general case, for the right-angled case
we will succeed in computing shortlex normal forms in logspace.

Recall that a \emph{right-angled Coxeter group} is specified by a
finite undirected graph $(\Sigma,I)$, i.e., an independence alphabet.
The set $\Sigma$ is the generating set and the relations are 
$a^2=1$ for all $a \in \Sigma$ and $ab=ba$ for all $(a,b) \in I$.
We denote this right-angled Coxeter group by $C(\Sig, I)$.
Similarly to the graph group $G(\Sig, I)$, the right-angled Coxeter group
$C(\Sig,I)$ can be defined by a (strongly) confluent and length-reducing trace
rewriting system (this time on $M(\Sigma,I)$ instead of $M(\Gamma,I_\Gamma)$). Let 
$$
S_C = \{ a^2 \to 1 \mid a \in \Sigma \} .
$$
Then $S_C$ is indeed (strongly) confluent and  length-reducing on $M(\Sigma,I)$ and 
the quotient $M(\Sigma,I)/S_C$ is $C(\Sigma,I)$.
Hence we have two closely related (strongly) confluent and 
length-reducing trace rewriting systems: $S_G$ defines
the graph group $G(\Sig, I)$ and $S_C$ defines
the \raCG $C(\Sig, I)$. Both systems  define unique normal forms
of geodesic length:  $\widehat u \in M(\Gam, I_\Gamma)$ for $S_G$ and $\widehat u \in M(\Sig, I)$ for $S_C$.
Note that there are no explicit commutation rules as they are
\emph{built-in} in trace theory.
There is a linear time algorithm for computing
$\widehat u$; see \cite{die90} for a more general result of this type. 
 
It is well known that a graph group $G(\Sigma,I)$ can be embedded into
a \raCG \cite{hw99}. For this, one has to duplicate each letter from $\Sigma$.
Formally, we can take the \raCG $C(\Gam, I_\Gamma)$ (in which $\ov a$ 
does not  denote the inverse of $a$).
Consider the mapping $\phi(a) = a\ov a$ {}from $\Gam$ to $\Gam^*$.
Obviously, $\phi$ induces a homomorphism from $G(\Sig, I)$ to the Coxeter group 
$C(\Gam, I_\Gamma)$. As $\IRR(S_G) \subseteq M(\Gamma,I_\Gamma)$ is mapped to 
$\IRR(S_C) \sse M(\Gam, I_\Gamma)$,  we recover the well-known fact that 
$\phi$ is injective. Actually we see more. Assume that 
$\wh w$ is the shortlex normal form of some $\phi(g)$. Then 
replacing in $\wh w$
factors $a\ov a$ with $a$ and replacing factors $\ov a a$
with $\ov a$ yields a logspace reduction of the problem of computing shortlex normal forms
in graph groups to the problem of computing shortlex normal forms
in \raCG{}s. Thus, for our purposes it is enough to calculate 
shortlex normal forms for {\raCG}s of type $C(\Sig, I)$ in logspace.
For the latter, it suffices to compute in logspace  on input $u \in \Sig^*$ 
some trace (or word)  $v$ such that $u = v$ in $C(\Sig, I)$
and $\abs{v} = \Abs u$. Then, the shortlex normal form for $u$ is the
lexicographic normal form of the trace $v$, which can be easily
computed in logspace from $u$. 

A  trace in $M(\Sig, I)$ is called a {\em Coxeter-trace}, if it does not have any 
factor $a^2$ where $a \in \Sig$.
It follows that every element in $C(\Sig, I)$ has a unique
representation as a \Ct. Let $a \in \Sig$. A trace $u$ is called $a$-short, if during the derivation
$u \RAS*{S_C} \wh u \in \IRR(S_C)$ the rule $a^2 \ra 1$ is not applied. Thus, 
$u$ is  $a$-short \IFF the number 
of occurrences of the letter $a$ is the same in the trace $u$ and 
its \Ct $\wh u$.
We are interested in the set of letters which survive the 
reduction process.  By $\wh\alp(u) = \alp(\wh u)$ we denote the 
alphabet of the unique \Ct $\wh u$ with $u = \wh u$ in $C(\Sig, I)$. 
Here is a crucial observation: 

\begin{lemma}\label{lem:otto}
A trace $u$ is $a$-short \IFF $u$ has no 
factor $ava$ such that $\wh \alp( v) \sse I(a)$. 
\end{lemma}

\begin{proof}
 If $u$ contains a factor $ava$ such that $\wh \alp( v) \sse I(a)$,
then $u$ is clearly not $a$-short. We prove the other direction
by induction on the length of $u$.  
Write $u = a_1 \cdots a_n$ with $a_i\in \Sig$. We identify $u$ with its dependence graph  $\DG(u)$ which has vertex 
set $\oneset{1 \lds n}$. Assume that $u$ is  not $a$-short. Then,
during the derivation $u \RAS{*}{S_C} \wh u$, for a first time a vertex $i$ with label $a_i= a$ 
is canceled with vertex $j$ with label $a_j= a$ and $i <j$.
It is enough to show that $\wh\alpha(a_{i+1} \cdots a_{j-1}) \sse I(a)$.
If the cancellation of $i$ and $j$  happens in the first step of the
rewriting process, then $\alpha(a_{i+1} \cdots a_{j-1}) \subseteq
I(a)$ and we are done. So, let the first step cancel 
vertices $k$ and $\ell$ with labels $a_k = a_\ell= b$ and $k<\ell$.
Clearly, $\oneset{i,j}\cap \oneset{k,\ell} = \es$.
The set $\wh\alpha(a_{i+1} \cdots a_{j-1})$ can change, only if 
either $i<k<j< \ell$ or $k<i< \ell<j$. However in both cases we must have
$(b,a)  \in I$, and  we are done by induction.
\qed
\end{proof}
In the right-angled case, 
the standard geometric  representation (see \eqref{eq-geo-repr}) 
$\sig: C(\Sig, I) \to \GL(n,\Z)$ (where $n = |\Sigma|$)
can be defined as follows, where again we write 
$\sigma_a$ for the mapping $\sigma(a)$:
\begin{eqnarray*}
\sig_a(a) & = & -a, \\
\sig_a(b) & = & \begin{cases}
      b  & \text{ if $(a,b) \in I$,} \\
  b +2a  & \text{ if $(a,b) \in D$ and $a \neq b$.}
\end{cases}
\end{eqnarray*}
In this definition, $a,b$ are letters. We identify  $\Z^n = \Z^{\Sig}$ 
and vectors from $\Z^n$ are written 
as formal sums $\sum_{b} \lambda_b b$. 
One can easily verify that $\sigma_{ab}(c) = \sigma_{ba}(c)$ for
$(a,b) \in I$ and $\sigma_{aa}(b) = b$. 
Thus, $\sigma$ defines indeed a homomorphism from $C(\Sigma,I)$ to
$\GL(n,\Z)$ (as well as homomorphisms from $\Sig^*$ and {}from  $M(\Sigma,I)$ to
$\GL(n,\Z)$).
Note that if $w = uv$ is a trace and $(b,v) \in I$ for a symbol $b$,
then $\sigma_w(b) = \sigma_u(b)$.
The following proposition is fundamental 
for understanding how the internal structure of $w$ is reflected by letting $\sig_w$ act on letters
(called \emph{simple roots} in the literature). 
For lack of a reference for this variant (of a  well-known general fact) and since the 
proof is rather easy in the right-angled case
(in contrast to the general case), we give a proof, which is purely 
combinatorial.

\begin{proposition}\label{prop:hugo}
Let $wd$ be a \Ct, $\sig_w(d) = \sum_{b} \lambda_b b$ and $wd = udv$ where 
$ud$ is prime and $(d,v) \in I$.
Then it holds: 
\begin{enumerate}[(1)]
\item $\lambda_b \neq 0 \iff b \in \alp(ud)$. Moreover,  $\lambda_b > 0$ for all $b \in \alp(ud)$.
\item Let $b,c \in \alp(ud)$, $b \neq c$, and assume that 
the first $b$ in $\DG(ud)$ appears before the first $c$ in $\DG(ud)$.  
Then we have $\lambda_b > \lambda_c >0$. 
\end{enumerate}
\end{proposition}

\begin{proof}
We prove both statements of the lemma by induction on $\abs{u}$.
For $\abs{u}= 0$ both statements are clear. Hence, let
$u = a u'$ and $\sig_{u'}(d)=  \sum_{b} \mu_b b$.
Thus,  
$$\sig_{u}(d)=  \sum_{b} \lambda_b b = \sig_{a}(\sum_{b} \mu_b b)
= \sum_{b} \mu_b \sig_{a}(b).$$
Note that $\mu_b =  \lambda_b$ for all $b \neq a$.
Hence, by induction  $\lambda_b =  0$ for all $b \notin \alp(ud)$
and $\lambda_b > 0$ for all $b \in \alp(ud) \setminus \{a\}$.

Let us now prove (2) for the trace $u$ (it implies $\lambda_a > 0$ and
hence (1)).
Consider $b,c \in \alp(ud)$, $b \neq c$, such that the first $b$ in
$\DG(ud)$  appears before the first $c$ in $\DG(ud)$. Clearly, this implies $c \neq a$. For $b \neq a$
we obtain that the first $b$ in $\DG(u'd)$ appears before the first $c$ in
$\DG(u'd)$. Hence, by induction we get $\mu_b > \mu_c > 0$. Claim (2) follows
since $b \neq a \neq c$ implies
$\mu_b =  \lambda_b$ and $\mu_c =  \lambda_c$.

Thus, let $a=b$. As there is path from the first $a$ to every $c$ in 
$\DG(ud)$ we may replace $c$ by the first letter we meet on such a path.
Hence we may assume that $a$ and $c$ are dependent. 
Note that $a \neq c$ because $u$ is a \Ct. 
Hence, $\lambda_c = \mu_c > 0$ and 
it is enough to show $\lambda_a > \mu_c$. But 
$\lambda_a \geq 2 \mu_c - \mu_a$ by the definition of $\sigma_a$.
If $\mu_a= 0$, then $\lambda_a \geq 2 \mu_c$, which implies 
$\lambda_a > \mu_c$, since $\mu_c > 0$.
Thus, we may assume $\mu_a> 0$. By induction, we get $a \in \alp(u'd)$. Here comes the
crucial point: the first $c$ in $\DG(u'd)$ must appear before the first $a$
in $u'd$. Thus, $\mu_c > \mu_a$ by induction, which finally implies
$\lambda_a \geq 2 \mu_c - \mu_a = \mu_c + (\mu_c - \mu_a) > \mu_c$.
 \qed
\end{proof} 

\begin{corollary}\label{cor:nixoderwas}
Let $C(\Sig,I)$ be a fixed \raCG. 
Then on input $w \in \Sig^*$ we can calculate in logspace 
the alphabet $\wh\alpha(w)$ of the corresponding \Ct $\wh w$. 
\end{corollary}

\begin{proof}
 Introduce a new letter $x$ which depends on all other letters from $\Sigma$. 
 We have $\sig_w(x)= \sig_{\wh w}(x) = \sum_{b} \lambda_b b$. As $\wh
 wx $ is a \Ct and prime, we have for all $b \in \Sigma$: 
 $$b \in
 \wh\alpha(w) \Longleftrightarrow b \in \alpha(\wh w x)
 \Longleftrightarrow \lambda_b \neq 0,
 $$ 
where the last equivalence follows from 
\refprop{prop:hugo}. Whether $\lambda_b \neq 0$ can be checked in
logspace, by computing $\lambda_b \;\mathrm{mod}\; m$ for all numbers 
$m \leq \abs w$, since the least common multiple of the first $n$ numbers is larger than 
$2^n$ (if $n \geq 7$) and the $\lambda_b$ are integers with 
$\abs{\lambda_b} \leq 2^{\abs{w}}$. See also \cite{lz77} for 
an analogous statement in the general context of linear groups. 
\qed
\end{proof}
The hypothesis in \refcor{cor:nixoderwas} being a \raCG is actually not necessary as we have seen in  \refthm{thm:allcoxoderwas}. It remains open whether this hypothesis can be removed
in the following theorem. 

\begin{theorem}\label{thm:lexshort}
Let $G$ be a fixed graph group or a fixed \raCG. 
Then we can calculate in logspace shortlex normal forms in $G$. 
\end{theorem}

\begin{proof}
As remarked earlier, 
it is enough to consider a \raCG $G = C(\Sig,I)$. 
Fix a letter $a \in \Sigma$. We first construct a logspace
transducer, which computes for an input trace $w \in M(\Sigma,I)$
a trace $u \in M(\Sigma,I)$ with the following properties:
(i) $u = w$ in $C(\Sigma,I)$, (ii) $u$ is $a$-short, and 
(iii) for all $b \in \Sigma$, if $w$ is $b$-short, then
also $u$ is $b$-short. 
Having such a logspace transducer for every $a \in \Sigma$,
we can compose all of them in an arbitrary order (note that
$|\Sigma|$ is a constant) to obtain a logspace transducer
which computes for a given input trace $w \in M(\Sigma,I)$ 
a trace $u$ such that $w = u$ in $C(\Sigma,I)$
and $u$ is $a$-short for all $a \in \Sigma$, i.e.,
$u \in \IRR(S_C)$. Thus $u = \wh w$. From $u$ we can compute easily in 
logspace the Hasse diagram of $\DG(u)$ and then the shortlex normal form.

So, let us fix a letter $a  \in \Sigma$ and an input trace $w= a_1
\cdots a_n$, where $a_1, \ldots, a_n \in \Sigma$.
We remove from left to right positions (or 
vertices) labeled by  
the letter $a$ which cancel and which therefore do not appear in $\wh
w$. We read $a_1 \cdots a_n$ from left to right.
In the $i$-th stage do the following: 
If $a_i \neq a$ output the letter $a_i$ and switch to the $(i+1)$-st stage.
If however 
$a_i = a$, then compute in logspace 
(using \refcor{cor:nixoderwas}) the maximal index $j>i$ (if it exists) such that
$a_j = a$ and $\wh \alp(a_{i+1} \cdots a_{j-1}) \sse I(a)$. 
If no such index $j$ exists, then append the letter $a_i$ to the
output tape and switch to the $(i+1)$-st stage. 
If $j$ exists, then append the word $a_{i+1} \cdots a_{j-1}$ to the
output tape, but omit all $a$'s. After that switch immediately to stage 
$j+1$. Here is a pseudo code description of the algorithm, where
$\pi_{\Sigma \setminus \{a\}} : \Sigma^* \to (\Sigma \setminus
\{a\})^*$ denotes the homomorphism that deletes all occurrences of $a$.

\medskip
\noindent
$i := 1$; \\
$w := 1$ \hfill (the content of the output tape
of the transducer) \\
{\bf while} $i \leq n$ {\bf do} \\
\hspace*{.5cm} {\bf if} $a_i \neq a$ {\bf then} \\
\hspace*{1cm} $w := wa_i$; \\
\hspace*{1cm} $i := i+1$ \\
\hspace*{.5cm} {\bf else} \\
\hspace*{1cm} $j := \text{undefined}$\\
\hspace*{1cm} {\bf for} $k=i+1$ {\bf to} $n$ {\bf do} \\
\hspace*{1.5cm} {\bf if} $a_k = a$ and 
$\wh \alp(a_{i+1} \cdots a_{k-1}) \sse I(a)$ {\bf then}\\
\hspace*{2cm} $j := k$\\
\hspace*{1.5cm} {\bf endif} \\
\hspace*{1cm} {\bf endfor}\\
\hspace*{1cm} {\bf if}  $j = \text{undefined}$ {\bf then}\\
\hspace*{1.5cm} $w := wa_i$; \\
\hspace*{1.5cm} $i := i+1$ \\
\hspace*{1cm} {\bf else} \\
\hspace*{1.5cm} $w := w\,\pi_{\Sigma \setminus \{a\}}(a_i \cdots a_{j-1})$; \\
\hspace*{1.5cm} $i := j+1$ \\
\hspace*{1cm} {\bf endif} \\
\hspace*{.5cm} {\bf endif}\\
{\bf endwhile}\\
{\bf return}(w)
\medskip

\noindent
Let $w_s$ be the content of the output tape at the
beginning of stage $s$, i.e., when the algorithm checks
the condition of the while-loop and variable $i$ has value $s$. 
(hence, $w_1 = 1$ and $w_{n+1}$ is the final output).
The invariant of the algorithm is that  
\begin{itemize}
\item $w_s = a_1 \cdots a_{s-1}$ in $C(\Sigma,I)$,
\item $w_s$ is $a$-short, and 
\item if $a_1 \cdots a_{s-1}$ is $b$-short, then also
$w_s$ is $b$-short.
\end{itemize}
The proof of this fact uses \reflem{lem:otto}. 
\qed
\end{proof}

\section{Free partially commutative inverse monoids}\label{sec:fp}

A monoid $M$ is \emph{inverse}, if
for every $x \in M$ there is  
$\inv{x} \in M$ 
with 
\begin{equation} \label{INV 1,2,3}
  x \inv{x} x  = x,  \quad
  \inv{x} x \inv{x} = \inv{x}, \text{ and } \quad
x\inv{x}\, y\inv{y} = y\inv{y}\, x\inv{x}. 
\end{equation}
The element $\ov x$ is uniquely defined by these properties and it is called
the \emph{inverse} of $x$. Thus, we may also use the notation 
$\ov x = x^{-1}$. 
It is easy to see that every idempotent element in an 
inverse monoid has the form $xx^{-1}$, and all these elements are  idempotent.
Using equations \eqref{INV 1,2,3} 
for all $x,y \in \Gam^*$  as defining relations 
we obtain 
the \emph{free inverse monoid }Ê$\FIM(\Sigma)$ which has been widely studied
in the literature.  More details on inverse monoids can be found in 
\cite{Law99}.

An \emph{inverse monoid over an independence alphabet $(\Sigma,I)$} is an
inverse monoid $M$ together with a mapping
$\varphi: \Sigma \to M$ such that
$\varphi(a)\varphi(b) = \varphi(b)\varphi(a)$ and
$\ov {\varphi(a)}\varphi(b) = \varphi(b)\ov{\varphi( a)}$
for all $(a,b) \in I$. We define the 
{\em free partially  commutative inverse monoid over} $(\Sigma,I)$
as the quotient monoid
$$\FIM(\Sigma,I) =
  \FIM(\Sigma)/\{ab=ba, \inv{a}b=b\inv{a} \mid (a,b)\in I\}.$$
It is an inverse  monoid over $(\Sigma,I)$.
Da~Costa has studied $\FIM(\Sigma,I)$ in his PhD~thesis \cite{vel03}.
He proved  that $\FIM(\Sigma,I)$ has
a decidable word problem, but he did not show any complexity bound. 
The first upper complexity bound for the word problem is due to
\cite{DiekertLohreyMiller08}, where it was shown to be solvable 
in time $O(n\log(n))$ on a RAM.
The aim of this section is to show that the space complexity of 
the word problem of $\FIM(\Sigma,I)$ is very low, too. 

\begin{theorem} \label{space}
The word problem of $\FIM(\Sigma,I)$
can be solved in logspace.
\end{theorem}

\begin{proof}
For a word $u = a_1 \cdots a_n$ ($a_1, \ldots, a_n \in \Gamma$)
let $u_i \in M(\Gamma,I_\Gamma)$ ($1 \leq i \leq n$) be the trace 
represented by the prefix $a_1 \cdots a_i$ and define the 
following subset of the trace monoid $M(\Gamma,I_\Gamma)$.
\begin{equation} \label{M(u)}
M(u) = \bigcup_{i=1}^n 
\{ p \mid p \text{ is a prime prefix of }
\wh{u_i} \} \subseteq M(\Gamma,I_\Gamma) .
\end{equation}
(This set is a partial commutative  analogue of the 
classical notion of \emph{Munn tree} introduced in \cite{Munn:74}.)
It is shown in \cite[Sect. 3]{DiekertLohreyMiller08} that  for all
words $u,v \in \Gamma^*$, 
$u=v$ in
$\FIM(\Sigma,I)$ if and only if 
\begin{enumerate}[(a)]
\item $u = v$ in the graph group
$G(\Sigma,I)$ and 
\item $M(u) = M(v)$.
\end{enumerate}
Since $G(\Sigma,I)$ is linear, condition (a) can be checked in logspace \cite{lz77,Sim79}. 
For (b), it suffices to show that the set $M(u)$ from \eqref{M(u)}
can be computed in logspace from the word $u$ (then $M(u) = M(v)$
can be checked in logspace, since
the word problem for the trace monoid $M(\Gamma,I_\Gamma)$
belongs to uniform $\mathsf{TC}^0$ \cite{AlGa91}
and hence to logspace).
By  \refthm{thm:lexshort} we can compute in logspace a list
of all normal forms $\wh{u_i}$ ($1 \leq i \leq n$), where
$u_i$ is the prefix of $u$ of length $i$.
By composing this logspace transducer with a logspace transducer
for computing prime prefixes (see \reflem{lemma-prim-prefixes}), we obtain a logspace
transducer for computing the set $M(u)$.
\qed
\end{proof}

\section{Concluding remarks and open problems}

We have shown that shortlex normal forms can be computed in logspace
for graph  groups and right-angled Coxeter groups.
For general Coxeter groups, we are able to 
compute in logspace the length of the shortlex normal form
and the set of letters appearing in the shortlex normal form. 
For even Coxeter groups we can do better and enhance the general result since we can compute the Parikh-image of geodesics. 
An obvious open problem is, whether for every (even) Coxeter group
shortlex normal forms can be computed in logspace. We are tempted
to believe that this is indeed the case. A more general question is, whether
shortlex normal forms can be computed in logspace for automatic groups.
Here, we are more sceptical. It is not known whether the word problem 
of an arbitrary automatic group can be solved in logspace.  In \cite{Lo05ijfcs},
an automatic {\em monoid} with a $\mathsf{P}$-complete word problem
was constructed. In fact, it is even open, whether the word problem for 
a hyperbolic group belongs to logspace. The best current upper bound
is {\sf LOGCFL} \cite{Lo05ijfcs}. So, one might first try to lower this bound e.g.
to  {\sf LOGDCFL} (the class of all languages that are logspace-reducible to 
a deterministic context-free language).
M.~Kapovich pointed out that there are non-linear hyperbolic 
groups. Hence the fact that linear groups have logspace
word problems (\cite{lz77,Sim79}) does not help here.


\begin{thebibliography}{10}

\bibitem{BesBr97}
M.~Bestvina and N.~Brady.
\newblock Morse theory and finiteness properties of groups.
\newblock {\em Inventiones Mathematicae}, 129(3):445--470, 1997.

\bibitem{bjofra05}
A.~Bj{\"o}rner and F.~Brenti.
\newblock {\em Combinatorics of {C}oxeter groups}, volume 231 of {\em Graduate
  Texts in Mathematics}.
\newblock Springer, New York, 2005.

\bibitem{bo93springer}
R.~Book and F.~Otto.
\newblock {\em String-Rewriting Systems}.
\newblock Springer-Verlag, 1993.

\bibitem{BrinkHowlett93}
B.~Brink and R.~B. Howlett.
\newblock {A finiteness property and an automatic structure for {C}oxeter
  groups}.
\newblock {\em Math. Ann.}, 296:179--190, 1993.

\bibitem{cai92stoc}
J.-Y. Cai.
\newblock Parallel computation over hyperbolic groups.
\newblock In {\em Proc. 24th ACM Symp. on Theory of Computing, STOC 92}, pages
  106--115. ACM-press, 1992.

\bibitem{AlGa91}
J.~G. Carme~{\`{A}}lvarez.
\newblock The parallel complexity of two problems on concurrency.
\newblock {\em Inform.~Process.~Lett.}, 38(2):61--70, 1991.

\bibitem{Casselman94}
W.~A. Casselman.
\newblock {Automata to Perform Basic Calculations in {C}oxeter Groups}.
\newblock {\em C.M.S. Conference Proceedings}, 16, 1994.

\bibitem{ChDaLi01}
A.~Chiu, G.~Davida, and B.~Litow.
\newblock Division in logspace-uniform {$\rm NC\sp 1$}.
\newblock {\em Theoretical Informatics and Applications. Informatique
  Th\'eorique et Applications}, 35(3):259--275, 2001.

\bibitem{CrGoWi09}
J.~Crisp, E.~Godelle, and B.~Wiest.
\newblock The conjugacy problem in right-angled {{A}rtin} groups and their
  subgroups.
\newblock {\em Journal of Topology}, 2(3), 2009.

\bibitem{CrWi04}
J.~Crisp and B.~Wiest.
\newblock Embeddings of graph braid and surface groups in right-angled {A}rtin
  groups and braid groups.
\newblock {\em Algebraic \& Geometric Topology}, 4:439--472, 2004.

\bibitem{DaPra11}
S.~Datta and R.~Pratap.
\newblock Computing bits of algebraic numbers.
\newblock Technical report, arXiv.org, 2011.
\newblock \url{http://arxiv.org/abs/1112.4295}.

\bibitem{davis08}
M.~W. Davis.
\newblock {\em The geometry and topology of {C}oxeter groups}, volume~32 of
  {\em London Math. Soc. Monographs Series}.
\newblock Princeton University Press, Princeton, NJ, 2008.

\bibitem{die90}
V.~Diekert.
\newblock {\em Combinatorics on Traces}.
\newblock Number 454 in Lecture Notes in Computer Science. Springer-Verlag,
  Heidelberg, 1990.

\bibitem{dkl12latin}
V.~Diekert, J.~Kausch, and M.~Lohrey.
\newblock Logspace computations in graph groups and {C}oxeter groups.
\newblock Lecture Notes in Computer Science. Springer-Verlag, 2012.
\newblock \emph{To appear in Proc.{} LATIN'2012, Arequipa, Peru}.

\bibitem{DiekertLohreyMiller08}
V.~Diekert, M.~Lohrey, and A.~Miller.
\newblock Partially commutative inverse monoids.
\newblock {\em Semigroup Forum}, 77(2):196--226, 2008.

\bibitem{dm06}
V.~Diekert and A.~Muscholl.
\newblock Solvability of equations in free partially commutative groups is
  decidable.
\newblock {\em International Journal of Algebra and Computation},
  16:1047--1070, 2006.
\newblock Journal version of ICALP 2001, 543--554, LNCS 2076.

\bibitem{dr95}
V.~Diekert and G.~Rozenberg, editors.
\newblock {\em The Book of Traces}.
\newblock World Scientific, Singapore, 1995.

\bibitem{DLS}
C.~{Droms}, J.~{Lewin}, and H.~{Servatius}.
\newblock {The length of elements in free solvable groups}.
\newblock {\em Proc. Amer. Math. Soc.}, 119:27--33, 1993.

\bibitem{elder2010}
M.~Elder.
\newblock A linear-time algorithm to compute geodesics in solvable
  {B}aumslag-{S}olitar groups.
\newblock {\em Illinois Journal of Mathematics}, 54(1):109--128, 2010.

\bibitem{ElderElstonOstheimer11}
M.~Elder, G.~Elston, and G.~Ostheimer.
\newblock On groups that have normal forms computable in logspace, May 2011.
\newblock AMS Sectional Meeting, Las Vegas. Paper in preparation.

\bibitem{ElRe2010}
M.~Elder and A.~Rechnitzer.
\newblock Some geodesic problems in groups.
\newblock {\em Groups. Complexity. Cryptology}, 2(2):223--229, 2010.

\bibitem{ech92}
D.~B.~A. Epstein, J.~W. Cannon, D.~F. Holt, S.~V.~F. Levy, M.~S. Paterson, and
  W.~P. Thurston.
\newblock {\em Word Processing in Groups}.
\newblock Jones and Bartlett, Boston, 1992.

\bibitem{GhPe07}
R.~Ghrist and V.~Peterson.
\newblock The geometry and topology of reconfiguration.
\newblock {\em Advances in Applied Mathematics}, 38(3):302--323, 2007.

\bibitem{HeAlBa02}
W.~Hesse, E.~Allender, and D.~A.~M. Barrington.
\newblock Uniform constant-depth threshold circuits for division and iterated
  multiplication.
\newblock {\em Journal of Computer and System Sciences}, 65:695--716, 2002.

\bibitem{hw99}
T.~Hsu and D.~T. Wise.
\newblock On linear and residual properties of graph products.
\newblock {\em Michigan Mathematical Journal}, 46(2):251--259, 1999.

\bibitem{Jer11}
E.~Je{\u{r}}{\'{a}}bek.
\newblock Root finding with threshold circuits.
\newblock Technical report, arXiv.org, 2011.
\newblock \url{http://arxiv.org/abs/1112.4295}.

\bibitem{Law99}
M.~V. Lawson.
\newblock {\em Inverse Semigroups: The Theory of Partial Symmetries}.
\newblock World Scientific, 1999.

\bibitem{lz77}
R.~J. Lipton and Y.~Zalcstein.
\newblock Word problems solvable in logspace.
\newblock {\em Journal of the Association for Computing Machinery},
  24(3):522--526, 1977.

\bibitem{Litow2010}
B.~E. Litow.
\newblock On sums of roots of unity.
\newblock In {\em Proceedings of ICALP 2010}, volume 6198 of {\em Lecture Notes
  in Computer Science}, pages 420--425. Springer, 2010.

\bibitem{Lo05ijfcs}
M.~Lohrey.
\newblock Decidability and complexity in automatic monoids.
\newblock {\em International Journal of Foundations of Computer Science},
  16(4):707--722, 2005.

\bibitem{LohOnd07}
M.~Lohrey and N.~Ondrusch.
\newblock Inverse monoids: {D}ecidability and complexity of algebraic
  questions.
\newblock {\em Inf. Comput.}, 205:1212--1234, 2007.

\bibitem{Miller}
C.~F. {Miller III}.
\newblock {Decision problems for groups -- survey and reflections}.
\newblock In {\em Algorithms and Classification in Combinatorial Group Theory},
  pages 1--60. Springer, 1992.

\bibitem{Munn:74}
W.~D. {Munn}.
\newblock {Free inverse semigroups}.
\newblock {\em Proc. London Math. Soc.}, 29(3):385--404, 1974.

\bibitem{MyRoUsVe08}
A.~Myasnikov, V.~Roman'kov, A.~Ushakov, and A.Vershik.
\newblock The word and geodesic problems in free solvable groups.
\newblock {\em Transactions of the American Mathematical Society},
  362:4655--4682, 2010.

\bibitem{Papa}
{\Ch}.~{Papadimitriou}.
\newblock {\em Computation Complexity}.
\newblock Addison-Wesley, 1994.

\bibitem{PRaz}
M.~{Paterson} and A.~{Razborov}.
\newblock {The set of minimal braids is co-NP-complete}.
\newblock {\em J. Algorithms}, 12:393--408, 1991.

\bibitem{Sim79}
H.-U. Simon.
\newblock Word problems for groups and contextfree recognition.
\newblock In {\em Proceedings of Fundamentals of Computation Theory (FCT'79),
  Berlin/Wendisch-Rietz (GDR)}, pages 417--422. Akademie-Verlag, 1979.

\bibitem{vel03}
A.~A. Veloso~da Costa.
\newblock {\em $\Gamma$-Produtos de Mon{\'o}ides e Semigrupos}.
\newblock PhD thesis, Universidade do Porto, Faculdade de Ci{\^e}ncias, 2003.

\bibitem{Vollmer99}
H.~Vollmer.
\newblock {\em Introduction to Circuit Complexity}.
\newblock Springer, Berlin, 1999.

\bibitem{Waack81}
S.~Waack.
\newblock Tape complexity of word problems.
\newblock In F.~G{\'e}cseg, editor, {\em Proceedings of Fundamentals of
  Computation Theory (FCT'81)}, volume 117 of {\em Lecture Notes in Computer
  Science}, pages 467--471. Springer, 1981.

\end{thebibliography}

\newcommand{\Ju}{Ju}\newcommand{\Ph}{Ph}\newcommand{\Th}{Th}\newcommand{\Ch}{C%
h}\newcommand{\Yu}{Yu}\newcommand{\Zh}{Zh}

\end{document}